%% file: graphs.tex
\pdfoutput=1

\documentclass[12pt, a4paper]{amsart}

\usepackage[T1]{fontenc}						
\usepackage[utf8]{inputenc}						

\usepackage{fixltx2e}	

\usepackage{verbatim}	

\usepackage[italian, english]{babel}

\usepackage[babel=true, kerning=true, spacing=true, tracking=true]{microtype}

\usepackage[babel]{csquotes}
\usepackage[style=numeric, maxnames=5, hyperref, backend=bibtex]{biblatex}

\usepackage{braket}		
\usepackage{mathtools}	
\usepackage{amssymb}
\usepackage{amsthm}

\DeclareMathOperator{\Ker}{Ker}

\DeclareMathOperator{\RGB}{RGB}
\DeclareMathOperator{\UGB}{UGB}
\DeclareMathOperator{\Gr}{Gr}
\DeclareMathOperator{\supp}{supp}
\DeclareMathOperator{\Tor}{Tor}

\theoremstyle{plain}
\newtheorem{thm}{Theorem}[section]
\newtheorem{lemma}[thm]{Lemma}
\newtheorem{prop}[thm]{Proposition}
\newtheorem{cor}[thm]{Corollary}
\theoremstyle{definition}
\newtheorem{defin}[thm]{Definition}
\theoremstyle{remark}
\newtheorem{rem}[thm]{Remark}
\newtheorem{notat}[thm]{Notation}

\renewcommand{\labelenumi}{(\roman{enumi})}

\def\cocoa{{\hbox{\rm C\kern-.13em o\kern-.07em C\kern-.13em o\kern-.15em A}}}

\usepackage{lmodern}
\normalfont

\usepackage{tabularx}
\usepackage{booktabs}	

\usepackage{tikz-cd}
\usetikzlibrary{arrows}
\tikzset{commutative diagrams/.cd, arrow style=tikz, diagrams={>=angle 90, line width=tikzcdrule}} 

\usepackage{enumitem}

\setlist[itemize]{topsep=6.0pt plus 2.4pt minus 3.6pt, itemsep=3.0pt plus 1.5pt minus 0.6pt, parsep=3.0pt plus 1.5pt minus 0.6pt, leftmargin=29.37pt, listparindent=0.0pt, labelwidth=23.5pt}
\setlist[enumerate]{topsep=6.0pt plus 2.4pt minus 3.6pt, itemsep=3.0pt plus 1.5pt minus 0.6pt, parsep=3.0pt plus 1.5pt minus 0.6pt, leftmargin=29.37pt, listparindent=0.0pt, labelwidth=23.5pt}

\usepackage{hyperref}
\hypersetup{hidelinks, pdftitle={Toric ideals associated with gap-free graphs}, pdfauthor={Alessio D'Al\`i}}

\usepackage{pdfpages}

\begin{document}
	\author{Alessio D'Al\`i}
	\address{Dipartimento di Matematica, Universit\`a degli Studi di Genova, Via Dodecaneso 35, 16146 Genova, Italy}
	\email{dali@dima.unige.it}
	\title{Toric ideals associated with gap-free graphs}
	\date{\today}

\begin{abstract}
In this paper we prove that every toric ideal associated with a gap-free graph $G$ has a squarefree lexicographic initial ideal. Moreover, in the particular case when the complementary graph of $G$ is chordal (i.e. when the edge ideal of $G$ has a linear resolution), we show that there exists a reduced Gr\"obner basis $\mathcal{G}$ of the toric ideal of $G$ such that all the monomials in the support of $\mathcal{G}$ are squarefree. Finally, we show (using work by Herzog and Hibi) that if $I$ is a monomial ideal generated in degree $2$, then $I$ has a linear resolution if and only if all powers of $I$ have linear quotients, thus extending a result by Herzog, Hibi and Zheng.
\end{abstract}


         \maketitle
	\input{part_1}

	\clearpage
	\phantomsection
	\addcontentsline{toc}{chapter}{\bibname}
	\nocite{*}
	\printbibliography
\end{document}

%% file: part_1.tex

\section{Introduction}
Algebraic objects depending on combinatorial data have attracted a lot of interest among both algebraists and combinatorialists: some valuable sources to learn about this research area are the books by Stanley \cite{Stanley}, Villarreal \cite{Villarreal}, Miller and Sturmfels \cite{MillerSturmfels}, and Herzog and Hibi \cite{HerzogHibi}.
It is often a challenge to establish relationships between algebraic and combinatorial properties of these objects.

Let $G$ be a simple graph and consider its vertices as variables of a polynomial ring over a field $K$. We can associate with each edge $e$ of $G$ the squarefree monomial $M_e$ of degree 2 obtained by multiplying the variables corresponding to the vertices of the edge. With this correspondence in mind, we can now introduce some algebraic objects associated with the graph $G$:
\begin{itemize}
\item the \emph{edge ideal} $I(G)$ is the monomial ideal generated by $\{M_e \mid e\textrm{ is an edge of }G\}$;
\item the \emph{toric ideal} $I_G$ is the kernel of the presentation of the $K$-algebra $K[G]$ generated by $\{M_e \mid e\textrm{ is an edge of }G\}$.
\end{itemize}
An important result by Fr\"oberg \cite{Fr} gives a combinatorial characterization of those graphs $G$ whose edge ideal $I(G)$ admits a linear resolution: they are exactly the ones whose complementary graph $G^c$ is chordal.
Another strong connection between the realms of commutative algebra and combinatorics is the one which links initial ideals of the toric ideal $I_G$ to triangulations of the edge polytope of $G$, see Sturmfels's book \cite{Sturmfels} and the recent article by Haase, Paffenholz, Piechnik and Santos \cite{HPPS}. Furthermore, Gr\"obner bases of $I_G$ have been studied among others by Ohsugi and Hibi \cite{OHToric} and Tatakis and Thoma \cite{TTUniversal}. A necessary condition for $I_G$ to have a squarefree initial ideal is the normality of $K[G]$, which was characterized combinatorially by Ohsugi and Hibi \cite{OHPolytopes} and Simis, Vasconcelos and Villarreal \cite{SVV}. Normality, though, is not sufficient: Ohsugi and Hibi \cite{OHStar} gave an example of a graph $G$ such that $K[G]$ is normal but all possible initial ideals of $I_G$ are not squarefree.

An interesting class of graphs is the one consisting of the so-called \emph{gap-free graphs} (following Dao, Huneke and Schweig's notation in \cite{DHS}), i.e. graphs such that any two edges with no vertices in common are linked by at least one edge. Unfortunately, these graphs do not have a standard name in the literature. Just to name a few possibilities:
\begin{itemize}
\item graph theorists refer to gap-free graphs as ``2$K_2$-free graphs'' and so do Hibi, Nishiyama, Ohsugi and Shikama in \cite{HNOS};
\item Nevo and Peeva call them ``$C_4$-free graphs'' in \cite{Nevo} and \cite{NevoPeeva};
\item Ohsugi and Hibi use the phrase ``graphs whose complement is weakly chordal'' in \cite{OHIndispensable};
\item Corso and Nagel call bipartite gap-free graphs ``Ferrers graphs'' in \cite{CorsoNagel}.
\end{itemize} 

The main goal of this paper is to prove that the toric ideal $I_G$ has a squarefree lexicographic initial ideal, provided the graph $G$ is gap-free (Theorem \ref{onesteplin}): moreover, the corresponding reduced Gr\"obner basis consists of circuits. In the particular case when $I(G)$ has a linear resolution (Theorem \ref{linres}) we are actually able to prove that the reduced Gr\"obner basis $\mathcal{G}$ we describe consists of circuits such that all monomials (both leading and trailing) in the support of $\mathcal{G}$ are squarefree, thus extending a result of Ohsugi and Hibi \cite{OHCompressed} on multipartite complete graphs.

\renewcommand{\labelenumi}{(\alph{enumi})}

In \cite{HHZ} Herzog, Hibi and Zheng proved that the following conditions are equivalent:
\begin{enumerate}
\item $I(G)$ has a linear resolution;
\item $I(G)$ has linear quotients;
\item $I(G)^k$ has a linear resolution for all $k \geq 1$.
\end{enumerate}
It is quite natural to ask (see for instance the article by Hoefel and Whieldon \cite{HW}) whether these conditions are in turn equivalent to the fact that \begin{enumerate}\item[(d)] $I(G)^k$ has linear quotients for all $k \geq 1$.\end{enumerate} In Theorem \ref{linchar} we prove that this is indeed the case, as can be deduced from results in \cite{HerzogHibi}. Note that all the equivalences between conditions (a), (b), (c), (d) above hold more generally for monomial ideals generated in degree 2 which are not necessarily squarefree.

The computer algebra system $\cocoa$ \cite{CoCoA} gave us the chance of performing computations which helped us to produce conjectures about the behaviour of the objects studied.

\renewcommand{\labelenumi}{(\roman{enumi})}

\section{Notation and known facts}
First of all, let us fix some notation.
$K$ will always be a field and $G$ a simple graph with vertices $V(G) = \{1, \ldots, n\}$ and edges $E(G) = \{e_1, \ldots, e_m\}$. We can associate to each edge $e = \{i, j\}$ the degree 2 monomial (called \emph{edge monomial}) $M_e := x_ix_j \in K[x_1, \ldots, x_n]$ and hence we can consider the \emph{edge ideal} $I(G) := (M_{e_1}, \ldots, M_{e_m})$ and the subalgebra $K[G] := K[M_{e_1}, \ldots, M_{e_m}]$. In the following we will denote by $I_G$ the \emph{toric ideal associated with} $G$, i.e. the kernel of the surjection
\[\begin{array}{ccc}  K[y_1, \ldots, y_m] & \twoheadrightarrow & K[G] \\ y_i & \mapsto & M_{e_i} \end{array}\]
Since the algebraic objects we defined are not influenced by isolated vertices of $G$, we will always assume without loss of generality that $G$ does not have any isolated vertex.
We will now introduce some terminology and state some well-known results about toric ideals of graphs: for reference, see for instance \cite[Section 10.1]{HerzogHibi}.

A collection of (maybe repeated) consecutive edges \[\Gamma = \{\{v_0, v_1\},\{v_1, v_2\}, \ldots, \{v_{q-1}, v_q\}\}\] (also denoted by $\{v_0 \rightarrow v_1 \rightarrow v_2 \rightarrow \ldots \rightarrow v_{q-1} \rightarrow v_q\}$) is called a \emph{walk} of $G$. If $v_0 = v_q$, the walk is \emph{closed}. If $q$ is even (respectively odd), the walk is an even (respectively odd) walk. A \emph{path} is a walk having all distinct vertices; a \emph{cycle} is the closed walk most similar to a path, i.e. such that vertices $v_0, \ldots, v_{q-1}$ are all distinct. A \emph{bow-tie} is a graph consisting of two vertex-disjoint odd cycles joined by a single path. Given a walk $\Gamma$, we will denote by $|\Gamma|$ the subgraph of $G$ whose vertices and edges are exactly the ones appearing in $\Gamma$.
If no confusion occurs, we will often write walks in more compact ways, such as by decomposing them into smaller walks. If $\Gamma$ is a walk, $-\Gamma$ denotes the walk obtained from $\Gamma$ by reversing the order of the edges.

If $\Gamma = \{\{v_0, v_1\},\{v_1, v_2\}, \ldots, \{v_{2q-1}, v_{2q}\}\}$ is an even closed walk, one can associate with $\Gamma$ a binomial $b_{\Gamma} \in K[y_1, \ldots, y_m]$ in the following way: \[b_{\Gamma} := \prod_{i=1}^q y_{\{v_{2i-2}, v_{2i-1}\}} - \prod_{i=1}^q  y_{\{v_{2i-1}, v_{2i}\}},\] where, if $e \in E(G)$, by $y_e$ we mean the variable which is mapped to $M_e$ by the standard surjection.
A \emph{subwalk} $\Gamma'$ of $\Gamma$ is an even closed walk such that all even edges of $\Gamma'$ are also even edges of $\Gamma$ and all odd edges of $\Gamma'$ are also odd edges of $\Gamma$.
An even closed walk $\Gamma$ is called \emph{primitive} if it does not have any proper subwalk. The set of binomials corresponding to primitive walks of a graph $G$ coincides with the so-called \emph{Graver basis} of $I_G$ (see for instance \cite{Sturmfels}) and is denoted by $\Gr_G$.  

\begin{rem} \label{primitivepaint}
Note that, given a primitive walk $\Gamma$, one can paint -- using two colours -- the edges of $|\Gamma|$ so that those appearing in an even position in $\Gamma$ are assigned the same colour and those appearing in an odd position are assigned the other one. If an edge were assigned both colours, then the walk $\Gamma$ would not be primitive: deleting inside both monomials one instance of the variable corresponding to that edge, one could construct a proper subwalk of $\Gamma$.
\end{rem}

The \emph{support} of a binomial $b = u-v \in I_G$ is the union of the supports of the monomials $u$ and $v$, that is to say the variables that appear in $u$ and $v$. A binomial $b \in I_G$ is called a \emph{circuit} if it is irreducible and has minimal support, i.e. there does not exist $b' \in I_G$ such that $\supp(b') \subsetneq \supp(b)$. The set of circuits of $I_G$ is denoted by $\textrm{C}_G$.

Let $I$ be an ideal of $S := K[x_1, \ldots, x_n]$. A Gr\"obner basis $\mathcal{G}$ of $I$ with respect to a term order $\tau$ is called \emph{reduced} if every element of $\mathcal{G}$ is monic, the leading terms of $\mathcal{G}$ minimally generate $in_{\tau}(I)$ and no trailing term of $\mathcal{G}$ lies in $in_{\tau}(I)$. Such a basis is unique and is denoted by $\RGB_{\tau}(I)$. Generally speaking, changing the term order $\tau$ yields a different reduced Gr\"obner basis: we will denote by $\UGB(I)$ the \emph{universal Gr\"obner basis} of $I$, i.e. the union of all reduced Gr\"obner bases of $I$.

\begin{prop}[{\cite[Proposition 4.11]{Sturmfels}}] \label{inclusions}
One has that $\emph{C}_G \subseteq \UGB(I_G) \subseteq \Gr_G$.
\end{prop}

The second inclusion of Proposition \ref{inclusions} means that every reduced Gr\"obner basis $\mathcal{G}$ of $I_G$ consists of binomials coming from primitive walks of $G$. Consider the set of monomials (both leading and trailing) in such a basis: if they are all squarefree, we will say that $\mathcal{G}$ is \emph{doubly squarefree}.

Complete characterizations of both $\textrm{C}_G$ (Villarreal \cite{VillarrealCircuits}) and $\UGB(I_G)$ (Tatakis and Thoma \cite{TTUniversal}) are known. We recall the characterization of $\textrm{C}_G$ (using the phrasing in Ohsugi and Hibi's article \cite{OHCircuits}) as a reference.

\renewcommand{\labelenumi}{\arabic{enumi}.}

\begin{prop} \label{circuits}
A binomial $b \in I_G$ is a circuit of $G$ if and only if $b = b_{\Gamma}$, where $\Gamma$ is one of the following even closed walks:
\begin{enumerate}
\item an even cycle;
\item $\{C_1, C_2\}$ where $C_1$ and $C_2$ are odd cycles with exactly one common vertex;
\item $\{C_1, p, C_2, -p\}$ where $C_1$ and $C_2$ are vertex-disjoint odd cycles and $p$ is a path running from a vertex of $C_1$ to a vertex of $C_2$.
\end{enumerate}
\end{prop}

\begin{defin}
Let $I \subseteq S := K[x_1, \ldots, x_n]$ be a graded ideal generated in degree $d$.
\begin{itemize}
\item If the minimal free resolution of $I$ as an $S$-module is linear until the $k$-th step, i.e. $\Tor^S_i(I, K)_j = 0$ for all $i \in \{0, \ldots, k\}$, $j \neq i+d$, we say that $I$ is \emph{k-step linear}.
\item If $I$ is $k$-step linear for every $k \geq 1$, we say $I$ has a \emph{linear resolution}.
\item If $I$ is minimally generated by $f_1, \ldots, f_s$ and for every $1 < i \leq s$ one has that $(f_1, f_2, \ldots, f_{i-1}) :_S (f_i)$ is generated by elements of degree 1, then $[f_1, \ldots, f_s]$ is called a \emph{linear quotient ordering} and $I$ is said to have \emph{linear quotients}.
\item If $I = I(G)$ for some graph $G$ and $I$ has one of the properties above, we say that $G$ has that property.
\end{itemize}
\end{defin}

\begin{prop}[{\cite[Proposition 8.2.1]{HerzogHibi}}] \label{linquotres}
Let $I \subseteq K[x_1, \ldots, x_n]$ be a graded ideal generated in degree $d$. Then
\[I\textrm{ has linear quotients }\Rightarrow I\textrm{ has a linear resolution.}\]
\end{prop}

\renewcommand{\labelenumi}{(\alph{enumi})}

We now recall an important result by Herzog, Hibi and Zheng (\cite{HHZ}) about the connection between linear quotients and linear resolution in the case when $I$ is a monomial ideal generated in degree 2. Condition (d) below did not appear in the original paper: its equivalence to other conditions, though, can be obtained quickly using results in \cite{HerzogHibi}.

\begin{thm} \label{linchar}
Let $I \subseteq K[x_1, \ldots, x_n] $ be a monomial ideal generated in degree 2. Then the following conditions are equivalent:
\begin{enumerate}
\item $I$ has a linear resolution;
\item $I$ has linear quotients;
\item $I^k$ has a linear resolution for all $k \geq 1$;
\item $I^k$ has linear quotients for all $k \geq 1$.
\end{enumerate}
\begin{proof}
The implications $(c) \Rightarrow (a)$ and $(d) \Rightarrow (b)$ are obvious, while $(b) \Rightarrow (a)$ and $(d) \Rightarrow (c)$ follow from Proposition \ref{linquotres}. It is then enough to prove that $(a) \Rightarrow (d)$, but this follows at once from \cite[Theorems 10.1.9 and 10.2.5]{HerzogHibi} (since the lexicographic order $<_{\textrm{lex}}$ introduced in Theorem 10.2.5 is of the kind appearing in Theorem 10.1.9).
\end{proof}
\end{thm}

\begin{rem} \label{linquotedges}
Theorem 10.2.5 and the proof of Theorem 10.1.9 in \cite{HerzogHibi} (or, as an alternative, just the proof of the implication $(a) \Rightarrow (b)$ in Theorem 10.2.6) tell us also that, if $I$ is a monomial ideal of degree 2 having a linear resolution and $\{m_1, m_2, \ldots, m_s\}$ is a minimal set of monomial generators for $I$, then there exists a permutation $\sigma$ of $\{1, \ldots, s\}$ such that $[m_{\sigma(1)}, m_{\sigma(2)}, \ldots, m_{\sigma(s)}]$ is a linear quotient ordering for $I$. As a consequence, if $I$ is the edge ideal of some graph $G$ having a linear resolution, there exists a way of ordering the edge monomials so that they form a linear quotient ordering.
\end{rem}

We thank Aldo Conca for pointing out the following result:

\begin{prop} \label{linquot}
Let $f_1, \ldots, f_s$ be distinct homogeneous elements of degree $d$ in $S := K[x_1, \ldots, x_n]$ which are minimal generators for the ideal $(f_1, \ldots, f_s)$. The following conditions are equivalent:
\begin{enumerate}
\item $[f_1, \ldots, f_s]$ is a linear quotient ordering;
\item the ideal $(f_1, \ldots, f_i)$ is 1-step linear for all $i \leq s$.
\end{enumerate}

\begin{proof}
Let us prove that $(a) \Rightarrow (b)$. Let $i \leq s$. If $[f_1, \ldots, f_s]$ is a linear quotient ordering, than $[f_1, \ldots, f_i]$ is too and hence, by Proposition \ref{linquotres}, the ideal $(f_1, \ldots, f_i)$ has a linear resolution; in particular, it is 1-step linear.

To prove that $(b) \Rightarrow (a)$, let $i \in \{2, \ldots, s\}$. Consider the exact sequence \[0 \to \Ker\varepsilon \to S(-d)^i \xrightarrow{\varepsilon} (f_1, \ldots, f_i) \to 0,\] where $\varepsilon$ is the map which sends $e_j$ to $f_j$ for all $j \in \{1, \ldots, i\}$. Then, by hypothesis, $\Ker\varepsilon$ is generated in degree 1. Since $(f_1, \ldots, f_{i-1}) :_S (f_i)$ is isomorphic to the $i$-th projection of $\Ker\varepsilon$, we are done.
\end{proof}
\end{prop}

In what follows, we will denote by $G^c$ the \emph{complementary graph} of $G$, i.e. the graph which has the same vertex set of $G$ and whose edges are exactly the non-edges of $G$.

The next result by Eisenbud, Green, Hulek and Popescu proves that, in our context, the algebraic concept of $k$-step linearity can be characterized in a purely combinatorial manner.

\begin{prop}[{\cite[Theorem 2.1]{EGHP}}] \label{EGHPprop}
Let $G$ be a graph and let $k \geq 1$. The following conditions are equivalent:
\begin{itemize}
\item $G$ is $k$-step linear;
\item $G^c$ does not contain any induced cycle of length $i$ for any $4 \leq i \leq k+3$.
\end{itemize}
\end{prop}

As a corollary, we recover the important result by Fr\"oberg characterizing combinatorially graphs with a linear resolution.

\begin{cor} [\cite{Fr}]
Let $G$ be a graph. Then $G$ has a linear resolution if and only if $G^c$ is chordal, i.e. $G^c$ does not contain any induced cycle of length greater than or equal to $4$. 
\end{cor}

Following the notation in \cite{DHS}, we will call a graph $G$ \emph{gap-free} if for any $\{v_1, v_2\}$, $\{w_1, w_2\}$ in $E(G)$ (where $v_1, v_2, w_1, w_2$ are all distinct) there exist $i, j \in \{1, 2\}$ such that $\{v_i, w_j\} \in E(G)$. In other words, in a gap-free graph any two edges with no vertices in common are linked by at least a bridge.

\begin{rem} \label{gapfreeonesteplin}
It is easy to see that $G$ is gap-free if and only if $G^c$ does not contain any induced cycle of length 4. It then follows from Proposition \ref{EGHPprop} that $G$ is gap-free if and only if $G$ is 1-step linear.
\end{rem}

The following theorem holds more generally for affine semigroup algebras.

\renewcommand{\labelenumi}{\arabic{enumi}.}

\begin{thm} \label{hocstu}
Let $G$ be a graph.
\begin{enumerate}
\item \emph{(Hochster \cite{Hochster})} If $K[G]$ is normal, then it is Cohen-Macaulay.
\item \emph{(Sturmfels \cite[Proposition 13.15]{Sturmfels})} If $I_G$ admits a squarefree initial ideal with respect to some term order $\tau$, then $K[G]$ is normal (and hence Cohen-Macaulay).
\end{enumerate}
\end{thm}

\renewcommand{\labelenumi}{(\roman{enumi})}

The problem of normality of graph algebras (and, as a consequence, of edge ideals, see \cite[Corollary 2.8]{SVV}) was addressed and completely solved by Ohsugi and Hibi \cite{OHPolytopes} and Simis, Vasconcelos and Villarreal \cite{SVV}.
One of the main results they found is the following:

\begin{thm}
A connected graph $G$ is such that $K[G]$ is normal if and only if $G$ satisfies the \emph{odd cycle condition}, i.e. for every couple of disjoint minimal odd cycles $\{C_1, C_2\}$ in $G$ there exists an edge linking $C_1$ and $C_2$.
\end{thm}

Ohsugi and Hibi \cite{OHStar} also found an example of a graph $G$ such that $K[G]$ is normal but $in_{\tau}(I_G)$ is not squarefree for every choice of $\tau$, hence the condition in Theorem \ref{hocstu}.2 is sufficient but not necessary.

\begin{rem}
There is a strong connection between squarefree initial ideals of $I_G$ and unimodular regular triangulations of the edge polytope of $G$. To get more information about this topic, see \cite{Sturmfels} and the recent work \cite{HPPS}, in particular Section 2.4.
\end{rem}

\section{Results}
We start by stating a result about the shape of primitive walks. This is a modification of \cite[Lemma 2.1]{OHToric}: note that primitive walks were completely characterized by Reyes, Tatakis and Thoma in \cite[Theorem 3.1]{RTT} and by Ogawa, Hara and Takemura in \cite[Theorem 1]{OHT}. In the rest of the paper we will often talk of primitive walks of type (i), (ii), (iii) referring to the classification below.

\begin{lemma}\label{primitivewalks}
Let $\Gamma$ be a primitive walk. Then $\Gamma$ is one of these:
\begin{enumerate}
\item an even cycle;
\item $\{C_1, C_2\}$ where $C_1$ and $C_2$ are odd cycles with exactly one common vertex;
\item $\{C_1, p_1, C_2, p_2, \ldots, C_h, p_h\}$ where the $p_i$'s are paths of length greater than or equal to one and the $C_i$'s are odd cycles such that $C_{i \pmod h}$ and $C_{i+1 \pmod h}$ are vertex-disjoint for every $i$.
\end{enumerate}
\begin{proof}
Let $\Gamma$ be a primitive walk neither of type (i) nor (ii). Since $\Gamma$ is primitive, there exists a cycle $C_1$ inside $\Gamma$ (otherwise $\Gamma = \{p, -p\}$ where $p$ is a path and hence all edges of $p$ would appear both in odd and even position in $\Gamma$, thus violating the primitivity); moreover, since $\Gamma$ is not of type (i), $C_1$ has to be odd. Let $C_1 = \{v_1 \rightarrow v_2 \rightarrow \ldots \rightarrow v_{2k+1} \rightarrow v_1\}$; then \[\Gamma = \{v_1 \rightarrow v_2 \rightarrow \ldots \rightarrow v_{2k+1} \rightarrow v_1=u_0 \rightarrow u_1 \rightarrow u_2 \rightarrow u_3 \rightarrow \ldots\}.\]
Let $s \geq 1$ be the least integer such that $u_s$ coincides with one of the vertices in $\{u_0 = v_1, v_2, \ldots, v_{2k+1}, u_1, \ldots, u_{s-1}\}$.
\begin{itemize}
\item Suppose $u_s = v_i$ where $i \neq 1$. If $s = 1$ and $i \in \{2, 2k+1\}$, we get that the edge $\{v_1, v_i\}$ is both an even and an odd edge of $\Gamma$ (contradiction). In all other cases, paint the edges appearing in $\Gamma$ red and black alternately and note that, since $i \neq 1$, there are both a red and a black edge of $C_1$ starting from $v_i$. Then exactly one of $\{v_1 = u_0 \rightarrow u_1 \rightarrow \ldots \rightarrow u_{s-1} \rightarrow u_s = v_i \rightarrow v_{i+1} \rightarrow \ldots \rightarrow v_{2k+1} \rightarrow v_1 = u_0\}$ and $\{v_1 = u_0 \rightarrow u_1 \rightarrow \ldots \rightarrow u_{s-1} \rightarrow u_s = v_i \rightarrow v_{i-1} \rightarrow \ldots \rightarrow v_2 \rightarrow v_1 = u_0\}$ is an even closed subwalk, thus violating the primitivity of $\Gamma$. This gives us a contradiction.
\item Suppose $u_s = u_i$ where $i \in \{0, \ldots, s-2\}$ (since $G$ has no loops, $i \neq s-1$). Note that one actually has that $i < s-2$, since $i = s-2$ would imply that the edge $\{u_{s-1}, u_s\}$ is both an even and an odd edge of $\Gamma$ (contradiction). Therefore there exists a cycle $C_2 = \{u_s = u_i \rightarrow u_{i+1} \rightarrow \ldots \rightarrow u_{s-1} \rightarrow u_s = u_i\}$ disjoint from $C_1$ by construction. Since $\Gamma$ is primitive, $C_2$ must be odd; moreover, since $\Gamma$ is not of type (ii), one has that $i \neq 0$. This means that we have found a path $p_1 = \{u_0 \rightarrow u_1 \rightarrow \ldots \rightarrow u_i\}$ linking the odd cycles $C_1$ and $C_2$. We can now repeat the whole procedure starting from the cycle $C_2$ to find a path $p_2$ and an odd cycle $C_3$ disjoint from $C_2$ and so on, hence proving the claim in a finite number of steps. \qedhere
\end{itemize}
\end{proof}
\end{lemma}

\begin{rem}
The referee noted that an alternative proof of Lemma \ref{primitivewalks} may be given using \cite[Theorem 1]{OHT}.
\end{rem}

\begin{rem}
Note that, by Proposition \ref{circuits}, all binomials corresponding to primitive walks of type (i) and (ii) are circuits.
\end{rem}

\renewcommand{\labelenumi}{(\alph{enumi})}

\begin{notat}
Let $G$ be a graph with $m$ edges and let $\tau$ be a term ordering on $K[y_1, \ldots, y_m]$. With a slight abuse of notation, we will often say that $e \preceq_{\tau} e'$ instead of $y_e \preceq_{\tau} y_{e'}$ (where $e, e' \in E(G)$). Moreover, if $H$ is a subgraph of $G$ and $\tau$ is lexicographic, we will say that $e \in E(H)$ is the \emph{leading edge} of $H$ with respect to $\tau$ if $y_{e'} \preceq_{\tau} y_e$ for every $e' \in E(H)$.
\end{notat}

Next we introduce the main technical lemma of the paper. Note that, when dealing with the vertices $v_1, \ldots, v_s$ of a cycle, for the sake of simplicity we will often write $v_i$ instead of $v_{i \pmod s}$.

\begin{lemma}\label{mainlemma}
Let $\Gamma$ be a primitive closed walk of $G$ of type (iii) and let $\tau$ be a lexicographic term order on $K[y_1, \ldots, y_m]$. Let $e$ be the leading edge of $|\Gamma|$ with respect to $\tau$: by Lemma \ref{primitivewalks}, $e$ lies into a bow-tie $\{C_1, p, C_2\}$.
Let $C_1 = \{v_1 \rightarrow v_2 \rightarrow \ldots \rightarrow v_{2k+1} \rightarrow v_1\}$, $C_2 = \{v'_1 \rightarrow v'_2 \rightarrow \ldots \rightarrow v'_{2{\ell}+1} \rightarrow v'_1\}$ and let $v_1$ and $v'_1$ be the starting and ending vertices of the path $p$.
Suppose one of the following two conditions holds:
\begin{enumerate}
\item $e \in p$ and there exist $i, j$ such that $\tilde{e} := \{v_i, v'_j\} \in E(G)$, $\tilde{e} \prec_{\tau} e$, $\tilde{e} \neq \{v_1, v'_1\}$;
\item $e = \{v_i, v_{i+1}\}$ and there exists $j$ such that at least one between $\{v_i, v'_j\}$ and $\{v_{i+1}, v'_j\}$ is an edge of $G$ (call it $\tilde{e}$) such that $\tilde{e} \prec_{\tau} e$ and $\tilde{e} \neq \{v_1, v'_1\}$.
\end{enumerate}
Then $b_{\Gamma} \notin \RGB_{\tau}(I_G)$.
\begin{proof}
First of all, by Remark \ref{primitivepaint} the primitivity of the walk $\Gamma$ allows us to paint the edges of $|\Gamma|$ red and black so that no two edges consecutive in $\Gamma$ are painted the same colour. We can assume without loss of generality that the edge $e$ is black.
\begin{enumerate}
\item Paint $\tilde{e}$ red. We can suppose without loss of generality that $i \neq 1$: hence, exactly one of $\{v_{i-1}, v_i\}$ and $\{v_i, v_{i+1}\}$ is black. This means that exactly one of the two paths going from $v_i$ to $v_1$ along $C_1$ has its first edge painted black: let $w$ be this path. We now need to define a path $w'$ going from $v'_1$ to $v'_j$.
\begin{itemize}
\item If $j \neq 1$, exactly one of $\{v'_{j-1}, v'_j\}$ and $\{v'_j, v'_{j+1}\}$ is black. Applying the same reasoning as before, let $w'$ be the path going from $v'_1$ to $v'_j$ along $C_2$ having its last edge painted black. 
\item If $j = 1$ and the last edge of $p$ is red, let $w' = \{v'_1 \rightarrow v'_2 \rightarrow \ldots \rightarrow v'_{2k+1} \rightarrow v'_1\}$ (in other words, the whole cycle $C_2$); if the last edge of $p$ is black, let $w'$ be the empty path in $v'_1$.
\end{itemize}
Let $\Gamma' = \{v'_j \xrightarrow{\tilde{e}} v_i \xrightarrow{w} v_1 \xrightarrow{p} v'_1 \xrightarrow{w'} v'_j\}$. By construction, $\Gamma'$ is an even closed walk, since its edges are alternately red and black and the first and the last one have different colours. Moreover, it is easy to check that $\Gamma'$ is primitive either of type (ii) (when $j=1$ and the last edge of $p$ is red) or of type (i) (in all other cases); hence, $\Gamma' \in \Gr_G$. Finally, since $\tau$ is a lexicographic term order, to get who the leading monomial of $b_{\Gamma'}$ is we just have to identify the leading edge of $\Gamma'$: since $\tilde{e} \prec_{\tau} e$ by hypothesis and the rest of the edges of $\Gamma'$ are edges of $\Gamma$, we get that the leading monomial of $b_{\Gamma'}$ is the one formed by black edges. Since the black edges of $\Gamma'$ all lie in $\Gamma$, we have that $in_{\tau}(b_{\Gamma'})$ divides $in_{\tau}(b_{\Gamma})$. Since $b_{\Gamma} \neq b_{\Gamma'}$, we have that $b_{\Gamma} \notin \RGB_{\tau}(I_G)$.
\item Paint $\tilde{e}$ red and define $w'$ in the same way as in part (a). Let $w$ be defined the following way:
\begin{itemize}
\item if $\tilde{e} = \{v_i, v'_j\}$, let $w := \{v_{i+1} \rightarrow v_{i+2} \rightarrow \ldots \rightarrow v_{2k+1} \rightarrow v_1\}$ (if $i = 2k+1$, $w$ is the empty path); 
\item if $\tilde{e} = \{v_{i+1}, v'_j\}$, let $w := \{v_{i} \rightarrow v_{i-1} \rightarrow \ldots \rightarrow v_2 \rightarrow v_1\}$ (if $i = 1$, $w$ is the empty path).
\end{itemize}
Let \[\Gamma' :=  \left\{ \begin{array}{lll} \{v_{i+1} \xrightarrow{w} v_1 \xrightarrow{p} v'_1 \xrightarrow{w'} v'_j \xrightarrow{\tilde{e}} v_i \xrightarrow{e} v_{i+1}\} & \textrm{if} & \tilde{e} = \{v_i, v'_j\} \\ \{v_i \xrightarrow{w} v_1 \xrightarrow{p} v'_1 \xrightarrow{w'} v'_j \xrightarrow{\tilde{e}} v_{i+1} \xrightarrow{e} v_i\} & \textrm{if} & \tilde{e} = \{v_{i+1}, v'_j\}\end{array} \right.\]
Reasoning the same way as in part (a), we get that $\Gamma'$ is an even closed walk; moreover, it can be easily checked that $\Gamma'$ is primitive either of type (ii) (when $v'_1$ belongs to $\tilde{e}$ and the last edge of $p$ is red or when $v_1$ belongs to $\tilde{e}$, with no restrictions on the colour of the last edge of $p$) or type (i) (in all other cases), hence $b_{\Gamma'} \in \Gr_G$. For the same reasons as in part (a), we get that $b_{\Gamma} \notin \RGB_{\tau}(I_G)$. \qedhere
\end{enumerate}
\end{proof}
\end{lemma}

\renewcommand{\labelenumi}{(\roman{enumi})}

\begin{thm} \label{linres}
Let $G$ be a graph with linear resolution and let $[e_1, \ldots, e_m]$ be an ordering of the edges of $G$ such that $[M_{e_1}, \ldots, M_{e_m}]$ is a linear quotient ordering for $I(G)$ (such an ordering exists by Remark \ref{linquotedges}). Let $\tau$ be the lexicographic order on $K[y_1, \ldots, y_m]$ such that $y_1 \prec_{\tau} y_2 \prec_{\tau} \ldots \prec_{\tau} y_m$. Then the reduced Gr\"obner basis of $I_G$ with respect to $\tau$ is doubly squarefree.
\begin{proof}
By Proposition \ref{linquot}, the linear quotient property is equivalent to asking that each subgraph $\{e_1, \ldots, e_i\}$ is 1-step linear, that is to say gap-free by Remark \ref{gapfreeonesteplin}. Let $\Gamma$ be a primitive walk such that at least one of the two monomials of $b_{\Gamma}$ is not squarefree. This implies that $\Gamma$ is primitive of type (iii). Hence, by Lemma \ref{primitivewalks}, we know that the leading edge $e$ of $\Gamma$ lies into a bow-tie $\{C_1, p, C_2\}$. Let $G_{\preceq e}$ be the subgraph of $G$ obtained by considering all the edges $e'$ such that $e' \preceq_{\tau} e$. This means that $G_{\preceq e} = \{e_1, e_2, \ldots, e_s = e\}$; hence, $G_{\preceq e}$ is gap-free. Using the notation of Lemma \ref{mainlemma}, we have to consider two different cases.
\begin{itemize}
\item If $e \in p$, consider the edges $\{v_2, v_3\}$ and $\{v'_2, v'_3\}$. Since $G_{\preceq e}$ is gap-free, there exists an edge $\tilde{e} \in E(G)$ which links the edges we are considering and is such that $\tilde{e} \prec_{\tau} e$. By applying Lemma \ref{mainlemma}.(a), we get that $b_{\Gamma} \notin \RGB_{\tau}(I_G)$.
\item If $e = \{v_i, v_{i+1}\}$, consider the edge $\{v'_2, v'_3\}$. Reasoning as before, we discover the existence of an edge $\tilde{e} \in E(G)$ linking these two edges and having the property that $\tilde{e} \prec_{\tau} e$: hence, by applying Lemma \ref{mainlemma}.(b), we get that $b_{\Gamma} \notin \RGB_{\tau}(I_G)$.
\end{itemize}
This ends the proof.
\end{proof}
\end{thm}

As a corollary we recover a result by Ohsugi and Hibi \cite{OHCompressed} about complete multipartite graphs:
\begin{cor}[\cite{OHCompressed}]
If $G$ is a complete multipartite graph, then there exists a doubly squarefree Gr\"obner basis of $I_G$.
\end{cor}
\begin{proof}
The complementary graph of a complete multipartite graph is a disjoint union of cliques and hence is chordal. Applying Theorem \ref{linres} yields the thesis. \endproof
\end{proof}

\begin{rem}
In Theorem \ref{linres} we actually proved that $\RGB_{\tau}(I_G)$
does not contain any binomials corresponding to primitive walks of type (iii). This means in particular that $\RGB_{\tau}(I_G)$ consists entirely of circuits (and hence $I_G$ is generated by circuits, as one could have already noticed applying Theorem 2.6 in Ohsugi and Hibi's article \cite{OHToric}).
\end{rem}

\begin{thm} \label{onesteplin}
Let $G$ be a gap-free graph and order its edges the following way: $\{v_{i_1}, v_{i_2}\} \preceq \{v_{j_1}, v_{j_2}\}$ if and only if $v_{i_1}v_{i_2} \preceq_{\sigma} v_{j_1}v_{j_2}$, where $\sigma$ is an arbitrary graded reverse lexicographic order on $K[v_1, \ldots, v_n]$. Rename the edges so that $e_1 \succ e_2 \succ \ldots \succ e_m$. Let $\tau$ be the lexicographic order on $K[y_1, \ldots, y_m]$ such that $y_1 \succ_{\tau} y_2 \succ_{\tau} \ldots \succ_{\tau} y_m$. Then $in_{\tau}(I_G)$ is generated by squarefree elements.
\begin{proof}
Let $\Gamma$ be a primitive walk of $G$ such that $in_{\tau}(b_{\Gamma})$ is not squarefree and let $e = \{u_1, u_2\}$ be the leading edge of $\Gamma$ with respect to $\tau$. Then, since $\Gamma$ has to be of type (iii), by Lemma \ref{primitivewalks} there exists a bow-tie $\{C_1, p, C_2\}$ containing $e$. We will use the notation of Lemma \ref{mainlemma} to denote the edges of this bow-tie.
\begin{itemize}
\item If $e \in C_1$, then no edges of $C_2$ have vertices in common with $e$. In the following we will say that a vertex $v \in V(|\Gamma|)$ satisfies condition $(<)$ if \[v \prec_{\sigma} u_1, v \prec_{\sigma} u_2.\]
Note that, by definition of $\sigma$ and $\tau$, if an edge $\{w_1, w_2\} \in E(|\Gamma|)$ shares no vertices with $e$, then at least one of $w_1$ and $w_2$ must satisfy condition $(<)$. Since no edges of $C_2$ share vertices with $e$, any pair of consecutive vertices in $C_2$ must include a vertex satisfying condition $(<)$: since $C_2$ is odd, by pigeonhole principle we get that there exists an edge $e'$ of $C_2$ whose vertices both satisfy condition $(<)$.

Since $G$ is gap-free, there exists $\tilde{e} \in E(G)$ linking $e$ and $e'$: moreover, since both vertices of $e'$ satisfy condition $(<)$, one has that $\tilde{e} \prec_{\tau} e$. If $\tilde{e} \neq \{v_1, v'_1\}$ then, by Lemma \ref{mainlemma}.(b), we get that $b_{\Gamma} \notin \RGB_{\tau}(I_G)$. If $\tilde{e} = \{v_1, v'_1\}$, then $v_1 \in e$ and we have to consider two different cases.

\begin{itemize}

\item[$\circ$] If $p$ is made of an even number of edges, then $\Gamma' := \{C_1, p, -\tilde{e}\}$ is a primitive walk of type (ii) such that $in_{\tau}(b_{\Gamma'})$ divides $in_{\tau}(b_{\Gamma})$. Hence $b_{\Gamma} \notin \RGB_{\tau}(I_G)$.

\item[$\circ$] If $p$ is made of an odd number of edges, then consider $\Gamma' := \{C_1, \tilde{e}, C_2, -\tilde{e}\}$. By Proposition \ref{circuits}, $b_{\Gamma'}$ is a circuit and hence $\Gamma'$ is a primitive walk. Since $in_{\tau}(b_{\Gamma'})$ is squarefree and divides $in_{\tau}(b_{\Gamma})$, we get that $b_{\Gamma} \notin \RGB_{\tau}(I_G)$.

\end{itemize}

\item If $e \in p$, we have to discuss two different situations.

If $p$ is made of more than one edge, then at least one of the cycles $C_1$ and $C_2$ has no vertices in common with $e$ (let it be $C_1$ without loss of generality). Then, applying the same pigeonhole reasoning used in the previous case, we discover the existence of an edge $e'$ of $C_1$ whose vertices both satisfy condition $(<)$. Since $G$ is gap-free, there exists $\tilde{e}$ linking $e'$ and $\{v'_2, v'_3\}$. Since $\tilde{e} \prec_{\tau} e$ by construction, applying Lemma \ref{mainlemma}.(a) we get that $b_{\Gamma} \notin \RGB_{\tau}(I_G)$.

The last case standing is the one where $p = \{\{v_1, v'_1\}\} = \{e\}$. Let $\hat{C}_1 := \{v_2, v_3, \ldots, v_{2k+1}\}$, $\hat{C}_2 := \{v'_2, v'_3, \ldots, v'_{2{\ell}+1}\}$. If there exist two consecutive vertices belonging to either $\hat{C}_1$ or $\hat{C}_2$ and satisfying condition $(<)$, then we can apply Lemma \ref{mainlemma}.(a) to infer that $b_{\Gamma} \notin \RGB_{\tau}(I_G)$.
Suppose otherwise. Then condition $(<)$ is satisfied alternately: to be more precise, we have that the vertices of $\hat{C}_1$ (or $\hat{C}_2$) satisfying condition $(<)$ are either the ones with odd index or the ones with even index. We can suppose without loss of generality that $v_3, v_5, \ldots, v_{2k+1}, v'_3, v'_5, \ldots, v'_{2{\ell}+1}$ are the vertices in $\hat{C}_1 \cup \hat{C}_2$ satisfying condition $(<)$. Consider the edges $\{v_2, v_3\}$ and $\{v'_2, v'_3\}$. Since $G$ is gap-free, these edges are surely linked by some edge $\tilde{e}$: if one of $v_3$ and $v'_3$ belongs to $\tilde{e}$ we have that $\tilde{e} \prec_{\tau} e$ and hence, by Lemma \ref{mainlemma}.(a), we can conclude that $b_{\Gamma} \notin \RGB_{\tau}(I_G)$.
What happens if $\tilde{e} = \{v_2, v'_2\}$? If $\tilde{e} \prec_{\tau} e$ we are done for the same reason as before. Suppose $\tilde{e} \succ_{\tau} e$. Then, by definition of $\tau$, at least one of $v_1$ and $v'_1$ (call it $w$) must be such that $w \prec_{\sigma} v_2$ and $w \prec_{\sigma} v'_2$. Since $e$ is the leading edge of $|\Gamma|$, though, one has that $e \succ_{\tau} \{v_1, v_2\}$ and $e \succ_{\tau} \{v'_1, v'_2\}$, hence $v'_1 \succ_{\sigma} v_2$ and $v_1 \succ_{\sigma} v'_2$. This gives us a contradiction. \qedhere
\end{itemize}
\end{proof}
\end{thm}

\begin{rem}
The proof of Theorem \ref{onesteplin} shows also that $\RGB_{\tau}(I_G)$ consists of circuits and hence (as we already knew by \cite[Theorem 2.6]{OHToric}) $I_G$ is generated by circuits. To see this, replace the hypothesis ``$\Gamma$ primitive walk such that $in_{\tau}(b_{\Gamma})$ is not squarefree'' with ``$\Gamma$ primitive walk of type (iii)'' and note that the only primitive walks of type (iii) that may appear in $\RGB_{\tau}(I_G)$ are bow-ties (more precisely, just those with a connecting path of length one). Since binomials associated with bow-ties are circuits by Proposition \ref{circuits}, we are done.
\end{rem}
\begin{rem}
In general, the construction appearing in Theorem \ref{onesteplin} does not necessarily yield a doubly squarefree reduced Gr\"obner basis of $I_G$. For instance, consider the gap-free graph $G$ with $6$ vertices and the following edges: \[\begin{split}&e_1 = \{1, 2\},~e_2 = \{1, 3\},~e_3 = \{2, 3\},~e_4 = \{1, 4\},~e_5 = \{3, 4\}, \\&e_6 = \{1, 5\},~e_7 = \{4, 5\},~e_8 = \{2, 6\},~e_9 = \{3, 6\},~e_{10} = \{5, 6\}.\end{split}\]
Note that the edges are ordered from the biggest to the smallest in a reverse lexicographic way according to the vertex order 1 > 2 > 3 > 4 > 5 > 6 (in the sense explained in the claim of Theorem \ref{onesteplin}).
Let $\tau$ be the lexicographic order on $K[y_1, \ldots, y_{10}]$ such that $y_1 \succ_{\tau} y_2 \succ_{\tau} \ldots \succ_{\tau} y_{10}$. Then \cocoa~computations yield \[\begin{split}\RGB_{\tau}(I_G) = \{&y_1y_{10} - y_6y_8,~y_1y_5 - y_3y_4,~y_1y_9 - y_2y_8,~y_5y_{10} - y_7y_9,~y_2y_7 - y_5y_6,\\&y_2y_{10} - y_6y_9,~y_3y_4y_{10} - y_5y_6y_8,~y_2y_5y_8 - y_3y_4y_9,~\boxed{y_3y_4y_7y_9 - y_5^2y_6y_8}\}.\end{split}\]
\end{rem}

\section*{Acknowledgements}
The author would like to thank his advisor Aldo Conca, Matteo Varbaro and the referee for their valuable comments and suggestions.